\documentclass{amsart}

\usepackage{enumerate}
\usepackage{xcolor}

\newtheorem{theorem}{Theorem}[section]
\newtheorem{lemma}[theorem]{Lemma}
\newtheorem{proposition}[theorem]{Proposition}

\theoremstyle{definition}

\theoremstyle{remark}
\newtheorem{remark}[theorem]{Remark}

\numberwithin{equation}{section}

\begin{document}

\title[Equivalence of solutions]{The equivalence of solutions to fractional and logarithmic Helmholtz equations}


%    author one information
\author[A.~Biswas]{Animesh Biswas}
\address{Department of Mathematics \\
Missouri State University \\ 
901 S. National Ave.\\
Springfield, MO 65897 \\
United States of America}
\curraddr{}
\email{ab7e@MissouriState.edu}
\thanks{}

%    author two information
\author[J.~Groszkiewicz]{Joseph Groszkiewicz}
\address{Department of Mathematics \\
Iowa State University \\
396 Carver Hall \\
Ames, IA 50011 \\
United States of America}
\curraddr{}
\email{josephg3@iastate.edu}
\thanks{}

%    author two information
\author[P.~R.~Stinga]{Pablo Ra\'ul Stinga}
\address{Department of Mathematics \\
Iowa State University \\
396 Carver Hall \\
Ames, IA 50011 \\
United States of America}
\curraddr{}
\email{stinga@iastate.edu}
\thanks{Research supported by Simons Foundation grant MP-TSM-00002709}

%    The 2020 edition of the Mathematics Subject Classification is
%    the current definitive version.

\date{}

\begin{abstract}
We show the following equivalence of solutions to general fractional Helmholtz equations: $u$ is a solution to 
$$A^su=\lambda^su$$
for \emph{some} fixed $0<s<2$ and some $\lambda>0$ if and only if $u$ is a solution to the same equation for \emph{all} $s\neq0$. Furthermore, we show as well that $u$ equivalently solves the logarithmic Helmholtz equation 
$$\log(A)u=(\log\lambda)u.$$
Here, $A$ is a nonnegative, linear operator on a Banach space $X$. In particular, our results hold whenever $A=-L$, where $L$
is the infinitesimal
generator of a $C_0$-semigroup in a Banach space $X$ satisfying mild assumptions that are typical in applications.

As a particular case, we recover known results for the fractional Laplacian $(-\Delta)^s$ in $\mathbb{R}^n$. More importantly, we provide a list of examples of other fractional power operators in settings where the Fourier transform is not available and for which our theorems apply. These include fractional powers and logarithms of second order elliptic operators in bounded domains, Laplace--Beltrami operators on Riemannian manifolds, space-time master equations, the discrete Laplacian, and fractional derivatives.
\end{abstract}

\maketitle

\section{Introduction}

Nonlocal equations of fractional order are central in many areas of pure and applied mathematics, including long-range interaction models, anomalous diffusions, mathematical models with memory effects and fractional calculus, see, for instance, \cite{Allen-Caffarelli-Vasseur, ACM, BRR, Caffarelli-Silvestre-Master, CHM, DL, GO2, GO, Stinga-Torrea-SIAM}.

It can be said that the foundational works of Caffarelli and Silvestre \cite{Caffa-Silv, Caffarelli-Silvestre-Master, CaffarelliSilvestre2009Regularity, CaffarelliSilvestre2011EvansKrylov, caffarelli2011regularity} have been the major force in the flourishing of the research activity on nonlocal equations in analysis, partial differential equations, numerical analysis, geometry and inverse problems, among other fields.

In recent years, another line of interesting research in analysis has been growing around nonlocal problems of ``order zero'' involving the logarithm $\log(-\Delta)$
of the Laplace operator $\Delta$ in the whole space $\mathbb{R}^n$, see \cite{boyadzhiev1994logarithms, CHEN2024110470, Chen-Weth, dyda2026dirichlet, hughes1980logarithm, JAROHS2020108732, nollau1969logarithmus, weilenmann1978continuity, yoshikawa1973logarithm} for new and classical results. This operator has been defined in the recent literature as the derivative of the fractional Laplacian $(-\Delta)^s$ with respect to the power $s$ evaluated at $s=0$, see \cite{Chen-Weth}.

It has been noted that, for the fractional Laplacian, regular solutions $u$ to the fractional Helmholtz equation $(-\Delta)^su=u$ in $\mathbb{R}^n$ for $0<s<1$ are given by the solutions to the classical Helmholtz equation $-\Delta u=u$ in $\mathbb{R}^n$, and that this result also extends to $1<s\leq 2$ and $s\in\mathbb{N}$, see \cite{Guan-Murugan} and also \cite{fall2016liouville}. This characterization was later extended to all $s>0$ in \cite{cheng2022equivalence}. The proofs of these results are based on the Fourier transform and the Caffarelli--Silvestre extension problem of \cite{Caffa-Silv}. In addition, it has been observed in \cite{lee2026fundamental} that a solution to the fractional Laplacian Helmholtz equation is also a solution to the logarithmic Laplacian equation $\log(-\Delta)u=0$. Characterizations of solutions to the equation $\Delta u+k^2u=0$ were previously studied for $k$ a real number by Agmon and H\"ormander \cite{MR466902}, and for $k$ a complex number by Agmon \cite{MR1039339}.

It turns out that, as we show in this paper, these equivalences of solutions are particular cases of a much more general property of fractional and logarithmic Helmoltz equations, and are independent of the availability of the Fourier transform.

In other words, we show that if $A$ is a nonnegative, linear operator on a Banach space $X$ then the statement
$$A^su=\lambda^su\qquad\hbox{for some}~0<s<2$$
and some $\lambda>0$, is equivalent to
$$A^su=\lambda^su\qquad\hbox{for all}~s>0.$$
Furthermore, we prove that, whenever $A^{-s}$
can be defined, the equivalence is also true with $-s$ in place of $s$. See Theorem \ref{thm:main theorem_eigen}.

We also relate these results to the equivalence of solutions with respect to the logarithmic equation
$$\log(A)u=(\log\lambda)u.$$
The logarithm of a nonnegative operator $A$ was first considered by Nollau \cite{nollau1969logarithmus}. Here, following \cite{nollau1969logarithmus} and \cite{martinez2001fractional},
we give two definitions of the logarithm of $A$. The first definition is given as the derivative of the positive fractional power operator $A^s$ with respect to $s$ evaluated at $s=0$. In this case, we can show that $A^su=\lambda^su$ implies that $\log(A)u=(\log\lambda)u$. Second, when $A$ is also one-to-one,
we define the logarithm of $A$ as the infinitesimal generator of the semigroup of the negative powers $\{A^{-s}\}_{s\geq0}$. We show that if $A$ is invertible then $A^{-s}u=\lambda^{-s}u$ if and only if $\log(A)u=(\log\lambda)u$. For these results, see Theorem \ref{thm:log1}.
In the case when $A=-L$, for $L$ the infinitesimal generator of a $C_0$-semigroup, we compare both definitions and obtain semigroup formulas for $\log(-L)$, see Propositions \ref{prop:logs} and \ref{prop:log-s}.

We then list a series of examples where our results apply. These include cases when $0$ is in the resolvent of the operator $L$ (so that $(-L)^{-s}$ can be defined) such as fractional elliptic operators in bounded domains, Laplace--Beltrami operators on closed Riemannian manifolds, and nonlocal space-time master operators based on bounded space domains;
as well as cases when $0$ is not in the resolvent set of $L$, such as the Laplacian on $\mathbb{R}^n$, the heat operator in the whole space, the discrete Laplacian, and the (left) derivative operator. The examples are given in Section \ref{section:examples}.

Our proofs rely on the definition and properties of fractional powers of nonnegative operators. Unlike previous works for the Laplacian \cite{cheng2022equivalence, Guan-Murugan}, we do not use the extension problem, although it is available in the literature for fractional powers of generators of $C_0$-semigroups in Banach spaces, see \cite{Gale, B-S}. However, we use the method of semigroups only when considering semigroup formulations for $\log(-L)$. The method of semigroups has been very fruitful in fractional nonlocal problems, including the analysis and application of the extension problem \cite{ B-S-1, B-DLC-S, B-S, Stinga-Torrea-SIAM}, numerical analysis \cite{Antil_Control, bonito2018numerical, Nocheto}, inverse problems \cite{Ghosh02122017}, fractional calculus \cite{STINGA2020111505}, just to mention a few.

Throughout this paper, $X$ is a Banach space with norm $\|\cdot\|_X$ and operator norm $\|\cdot\|$. We let $A:D(A)\subset X\to R(A)\subset X$ be a linear operator with domain $D(A)$, range $R(A)$, resolvent set $\rho(A)=\{\mu\in\mathbb{C}:(\mu I+A)^{-1}~\hbox{is bounded on}~X\}$, where $I$ denotes the identity operator on $X$, and spectrum $\sigma(A)=\mathbb{C}\setminus\rho(A)$. We say that $A$ is \textbf{nonnegative} if $(-\infty,0)$ is in the resolvent set $\rho(A)$ of $A$ and
$$\sup_{\mu>0}\|\mu(\mu I+A)^{-1}\|<\infty.$$
It turns out that if $A$ is one-to-one then $A^{-1}$ is also nonnegative. See \cite{martinez2001fractional} for more on nonnegative operators.

We also say that a family $\{S_t\}_{t\geq0}$ of bounded linear operators on $X$ is a \textbf{$C_0$-semigroup}
if $S_0=I$, $S_{t_1}\circ S_{t_2}=S_{t_1+t_2}$ for all $t_1,t_2\geq0$ and $\lim_{t\to0}S_tu=u$
for all $u\in X$. 
For a $C_0$-semigroup $\{S_t\}_{t\geq0}$, we define its infinitesimal generator $L$ as the linear operator given by
\begin{equation}\label{eq:infinitesimal}
Lu=\lim_{t\to0}\frac{S_tu-u}{t}
\end{equation}
with domain $D(L)=\{u\in X:\eqref{eq:infinitesimal}~\hbox{exists in}~X\}$. In this case, and from now on, we will write
$$S_t=e^{tL}.$$
Then $L$ is a closed operator and $D(L)$ is dense in $X$. Moreover, if $u\in D(L)$ then
$v(t)=e^{tL}u$ is the unique classical solution
$v\in C([0,\infty);X)\cap C^1((0,\infty);D(L))$
to the abstract heat equation
$$\begin{cases}
    \partial_tv=Lv&\hbox{for}~t>0\\
    v(0)=u.
\end{cases}$$
The semigroup $\{e^{tL}\}_{t\geq0}$ is said to be \textbf{uniformly bounded} if there is a constant $M>0$ such that $\|e^{tL}u\|_X\leq M\|u\|_X$ for all $u\in X$, $t\geq0$. In this case, it follows that $-L$ is a nonnegative operator. Finally, the semigroup $\{e^{tL}\}_{t\geq0}$ is \textbf{exponentially stable} if there are constants $M,\varepsilon>0$ such that $\|e^{tL}u\|_X\leq Me^{-\varepsilon t}\|u\|_X$, for all $u\in X$ and all $t\geq0$. See \cite{pazy1983semigroups} for more about semigroups.

Throughout the paper, $\Gamma$ denotes the classical Gamma function.

\section{Fractional Helmholtz equations}

In this section, $A$ is a nonnegative, linear operator on a Banach space $X$.

For the notions about nonnegative operators, their fractional powers and properties that will be introduced next, we follow \cite{MR930604} and \cite{martinez2001fractional}.

We define the \textbf{positive fractional powers} $A^s$ of a nonnegative operator $A$ for $s>0$. If $A$ is a bounded operator then $A^s=J^s$, where $J^s$ is the Balakrishnan operator defined as follows.
\begin{itemize}
    \item If $0<s<1$, $D(J^s)=D(A)$ and
    $$J^su=\frac{\sin(s\pi)}{\pi}\int_0^\infty \mu^{s-1}(\mu I+A)^{-1}Au\,d\mu.$$
    \item If $s=1$, $D(J^s)=D(A^2)$ and
    $$J^su=\frac{\sin(s\pi)}{\pi}\int_0^\infty \mu^{s-1}\bigg[(\mu I+A)^{-1}-\frac{\mu}{\mu^2+1}\bigg]Au\,d\mu+\sin(s\pi/2)Au.$$
    \item If $n<s<n+1$, $n\in\mathbb{N}$, $D(J^s)=D(A^{n+1})$ and
    $$J^su=J^{s-n}A^nu.$$
    \item If $s=n+1$, $n\in\mathbb{N}$, $D(J^s)=D(A^{n+2})$ and $$J^su=J^{s-n}A^nu.$$
\end{itemize}
The integrals above are understood in the sense of Bochner.

Next, if $A$ is unbounded and $0\in\rho(A)$ then we define $A^s=[(A^{-1})^s]^{-1}$. Finally, if $A$ is unbounded and $0\in\sigma(A)$ then $A^su=\lim_{\varepsilon\to0^+}(A+\varepsilon I)^su$ with the domain $D(A^s)$ being the set of all $u\in D((A+\varepsilon I)^s)$ for
all $\varepsilon>0$ close to zero for which the limit exists in $X$.

It turns out that $A^s$ is a closed linear operator, $A^n$ coincides with the $n$-times composition $A\circ\cdots\circ A$ of $A$, $n\in\mathbb{N}$, and satisfies the semigroup property $A^{s_1+s_2}=A^{s_1}\circ A^{s_2}$ for $s_1,s_2>0$. In particular,
$D(A^{s_0})\subset D(A^s)$ for all $0<s<s_0$. Finally, if $0<s<1$ then $A^s$ is also nonnegative and $(A^s)^r=A^{sr}$ for all $r>0$.

If $A$ is bounded then $A^s$ is also bounded and the function $s\mapsto A^su$ is continuous, for every $u\in X$.

Now we define the \textbf{negative fractional powers} of nonnegative operators $A$. If $A$ is one-to-one, so is $A^s$ and the inverse $(A^s)^{-1}$ of $A^s$ is given by
$$A^{-s}:=(A^{-1})^s.$$

Fix $s_0>0$. If $u\in D(A^{s_0})\cap R(A^{s_0})$ then the function $s\mapsto A^su$ is continuous for $-s_0<s<s_0$. If $u\in R(A^{s_0})$ then the function $s\mapsto A^su$ is continuous for $-s_0<s<0$.

The main result of this section is the equivalence of solutions for positive and negative fractional Helmholtz equations.

\begin{theorem}[Equivalence for fractional Helmholtz equations]\label{thm:main theorem_eigen}
    Let $A$ be a nonnegative operator on $X$. Let $\lambda>0$ and $u\in X$. The following are equivalent.
    \begin{enumerate}
    \item[$(1)$] $A^su=\lambda^su$ for some $0<s<2$.
    \item [$(2)$]$A^su=\lambda^su$ for all $s>0$.
    \end{enumerate}
If, in addition, $A$ is one-to-one, then $(1)$ and $(2)$ are also equivalent to the following two statements.
    \begin{enumerate}
    \item[$(3)$] $A^{-s}u=\lambda^{-s}u$ for some $0<s<2$.
    \item[$(4)$] $A^{-s}u=\lambda^{-s}u$ for all $s>0$.
    \end{enumerate}
\end{theorem}

The proof of Theorem \ref{thm:main theorem_eigen} is split into several lemmas.

\begin{lemma}\label{lem:half_power_eigen}
Let $0<s<2$. If $A^s u = \lambda^s u$ then $A^{s/2} u = \lambda^{s/2} u$.
\end{lemma}

\begin{proof}
    Since $u \in D(A^s)$, we also have $u \in D(A^{s/2})$. Let $v = A^{s/2} u$. Then
    $$ A^{s/2}v = A^s u = \lambda^s u.$$
    Define now $w = \lambda^{s/2} u- v\in D(A^{s/2})$. We see that
    $$ A^{s/2}w = A^{s/2}(\lambda^{s/2} u) - A^{s/2} v = \lambda^{s/2} v- \lambda^s u = -\lambda^{s/2} w.$$
    Thus, $(A^{s/2} + \lambda^{s/2} I) w = 0$. Since $0<s/2<1$, we have that $A^{s/2}$ is a nonnegative operator, and thus $(A^{s/2} + \lambda^{s/2}I)$ is one-to-one and, hence, $w=0$, which means that $A^{s/2}u=\lambda^{s/2}u$.
\end{proof}

\begin{lemma}\label{lem:double power}
    Let $s>0$. If $A^s u = \lambda^s u$ then $A^{ks} u = \lambda^{ks} u$ for all $k\in\mathbb{N}$.
\end{lemma}

\begin{proof}
    Let $f=A^s u$. Then $f= \lambda^s u$ and so $f \in D(A^s)$. Now,
$$A^s f = A^s (A^s u) =A^s (\lambda^s u) = \lambda^{2s}u,$$
which implies that $u \in D(A^{2s})$ and $A^{2s} u = \lambda^{2s} u$. By iteration, the general result for any $k\geq1$ follows.
\end{proof}

\begin{lemma}\label{lem:s less than 1}
    Let $0<s<2$. If $A^s u =\lambda^{s} u$ then $A^{\alpha} u = \lambda^{\alpha} u$ for every $0 <\alpha< \min \{s, 1\}$.
\end{lemma}

\begin{proof}
    By iteration of Lemma \ref{lem:half_power_eigen}, we have that $A^{s/2^k} u= \lambda^{s/2^k} u$ for any $k \in \mathbb{N}$. Now consider any $\alpha$ such that $0 <\alpha< \min \{s, 1\}$. Using the binary representation of real numbers, we can write
$$ \frac{\alpha}{s} = \lim_{n\to\infty}\sum^n_{k=1} \frac{c_k}{2^k}
=:\lim_{n\to\infty}\beta_n,$$
where $c_k =0$  or $c_k =1$.
Then we see that, by the semigroup property of positive fractional powers, for every $n$, 
$$A^{s\beta_n}u = \lambda^{s \beta_n}u.$$
The continuity with respect to the fractional power then implies that
$$A^\alpha u = \lim_{n\to\infty}A^{s\beta_n}u=\lim_{n\to\infty}\lambda^{s\beta_n}u=\lambda^\alpha u$$
as desired.
\end{proof}

\begin{lemma}\label{lem:alphalessthans}
    Let $0<s<2$. If $A^s u = \lambda^{s}u$ then $A^{\alpha} u = \lambda^{\alpha }u$ for every $0 <\alpha<s$.
\end{lemma}

\begin{proof}
    Assume first that $1< \alpha < 2$ and let $\gamma = \alpha /2$. Then $0<\gamma <1$ and $0< \gamma/s <1$. By Lemma \ref{lem:s less than 1}, $A^\gamma u = \lambda^{\gamma} u$. This gives 
    $$ A^\alpha u = A^{2 \gamma} u = A^\gamma(A^\gamma u)=\lambda^{2\gamma} u=\lambda^\alpha u. $$
\end{proof}
    
\begin{proof}[Proof of Theorem \ref{thm:main theorem_eigen}]
Assume that $A^s u = \lambda^s u$ for some fixed $0<s<2$. 
Lemma \ref{lem:alphalessthans} gives that $A^\alpha u = \lambda^{\alpha}u$ for every $0<\alpha<s$. Next choose $s_1$ such that $s<s_1$. Then there exists $\alpha<s$ such that $2^n \alpha = s_1$ for some $n \in \mathbb{N}$. But we know $A^\alpha u = \lambda^\alpha u$. Then using Lemma \ref{lem:double power}, we see that $A^{2^n \alpha} u =\lambda^{2^n \alpha} u$ or, which is the same, $A^{s_1} u = \lambda^{s_1}u$.
Thus $(1)$ implies $(2)$.

The equivalence with $(3)$ and $(4)$ follows immediately from the positive powers case, by applying $A^s$ to both sides of the equation $A^{-s}u=\lambda^{-s}u$ to get
$u=\lambda^{-s}A^su$.
\end{proof}

\begin{remark}
Theorem \ref{thm:main theorem_eigen} assumes the existence of an exponent $s\in (0,2)$ such that $A^s u=\lambda^s u$. This is due to our proof, because for Lemma \ref{lem:half_power_eigen} we need $A^{s/2}$ to be a nonnegative operator in order to conclude. We do not know whether the equivalence in Theorem \ref{thm:main theorem_eigen} remains valid under the weaker assumption that $A^s u=\lambda^s u$ for some arbitrary $s>0$.
\end{remark}

\begin{remark}
    It would also be interesting to extend Theorem \ref{thm:main theorem_eigen} to complex $s$ and $\lambda$. We do not focus on those cases since we are interested on the real case as this is the one that has typically appeared in previous applications.
\end{remark}

\section{Logarithmic Helmholtz equations}\label{section:log}

The logarithm of nonnegative, linear operators was first defined by Nollau \cite{MR261381}, see also \cite{martinez2001fractional}.

Let $A$ be a nonnegative, linear operator. Then the \textbf{logarithm} $\log(A)$ of $A$ is defined as
\begin{equation}\label{eq:logs}
\log(A)u=\frac{d}{ds}A^{s}u\bigg|_{s=0}=\lim_{s\to0^+}\frac{1}{s}(A^su-u)
\end{equation}
for $u\in D(\log(A))=\{u\in X:\eqref{eq:logs}~\hbox{exists in}~X\}$.
Observe that this way of defining $\log(A)$ is analogous to the elementary identity
\[
\log \lambda
=
\frac{d}{ds}\lambda^s\bigg|_{s=0}
\qquad\hbox{for}~\lambda>0.
\]
If, in addition, $A$ is one-to-one, then 
we have the equivalent formula
\begin{equation}\label{eq:log-s}
-\log(A)u=\frac{d}{ds}A^{-s}u\bigg|_{s=0}=\lim_{s\to0^+}\frac{A^{-s}u-u}{s},
\end{equation}
whenever $u\in D(A)\cap R(A)$, see \cite{martinez2001fractional}. Under the stronger assumption that $0\in\rho(A)$, we get that $A^{-s}$ is bounded with domain $D(A^{-s})=X$.
In fact, in this case, the family of operators $\{A^{-s}\}_{s\geq0}$ defines a $C_0$-semigroup on $X$. It then follows that $-\log(A)$ is simply the infinitesimal generator of such a semigroup, and it is therefore a closed operator with dense domain.

\begin{theorem}[Equivalence between fractional and logarithmic Helmholtz equations]\label{thm:log1}
Let $A$ be a nonnegative, linear operator on $X$. Fix $u\in X$ and $\lambda>0$. If
\begin{equation}\label{eq:forsome}
A^su =\lambda^s u\qquad\hbox{for some}~0<s<2
\end{equation}
then
\begin{equation}\label{eq:forlog}
  \log(A) u = (\log\lambda) u.
\end{equation}
If, in addition, $A$ is one-to-one, then
any of the following statements
\begin{enumerate}
    \item[$(1)$] $A^su =\lambda^s u$ for some $0<s<2$ (or, equivalently, all $s>0$);
    \item[$(2)$] $A^{-s}u =\lambda^{-s} u$ for some $0<s<2$ (or, equivalently, all $s>0$);
\end{enumerate}
imply \eqref{eq:forlog}. 

Finally, if $0\in\rho(A)$ then the converse statement holds, that is, \eqref{eq:forlog} implies both $(1)$ and $(2)$.
\end{theorem}

\begin{proof}
Suppose that \eqref{eq:forsome} holds. Then, by
Theorem \ref{thm:main theorem_eigen},
$A^su=\lambda^su$ for all $s>0$. Hence, by definition of the logarithm of $A$,
$$\log(A)u=u\lim_{s\to0}\frac{\lambda^s-1}{s}=(\log\lambda)u.$$

Assume that $A$ is one-to-one. Then, by Theorem \ref{thm:main theorem_eigen}, $(1)$ and $(2)$ are equivalent and, by the proof above, $(1)$ implies \eqref{eq:forlog}.

Finally, suppose that $0\in\rho(A)$ and that \eqref{eq:forlog} holds. It is enough to prove $(2)$ for all $s>0$. Since $u\in D(\log(A))$, the unique solution to the evolution equation
$$
\begin{cases}
    v_t = -\log(A)v&\hbox{for}~t>0 \\
    v(0) = u
\end{cases}
$$
is given by the semigroup generated by $-\log(A)$ acting on $u$, that is, $v(t)=e^{-t\log(A)}u=A^{-t}u$.
On the other hand, we verify that
$w(t)= \lambda^{-t} u$ also solves the same initial value problem. Indeed, $w(0)=u$ and
$$w_t = -(\log\lambda) \lambda^{-t} u = -\lambda^{-t} \log(A)u = - \log(A) w.
$$
Then, by uniqueness, $w=v$, which means that $A^{-t}u=\lambda^{-t}u$ for all $t\geq0$.
\end{proof}

We have essentially introduced two definitions of the logarithm of the nonnegative operator $A$, one in \eqref{eq:logs} and another one
(for the case when $A$ is also one-to-one) in \eqref{eq:log-s}. In order to compare both definitions, we assume that $A=-L$, where $L$ is the infinitesimal generator of a $C_0$-semigroup $\{e^{tL}\}_{t\geq0}$ on $X$.

\begin{proposition}[Formula for $\log(-L)$ for generators of uniformly bounded  $C_0$-semigroups]\label{prop:logs}
    Let $L$ be the infinitesimal generator a uniformly bounded $C_0$-semigroup $\{e^{tL}\}_{t\geq0}$ on $X$. Then, for any $u\in D(L)$,
    \begin{equation}\label{eq:logslimit}
    \log(-L)u=\lim_{s\to0}\frac{1}{\Gamma(1-s)}\int_0^\infty\frac{e^{-t}u-e^{tL}u}{t^{1+s}}\,dt.
    \end{equation}
    Therefore, at least formally,
    $$\log(-L)u=\int_0^\infty\frac{e^{-t}u-e^{tL}u}{t}\,dt.$$
\end{proposition}

\begin{proof}
For $u\in D(L)$ and $0<s<1$ we can write
$$(-L)^su=\frac{1}{\Gamma(-s)}\int_0^\infty\big(e^{tL}u-u\big)\,\frac{dt}{t^{1+s}},$$
where the integral is absolutely convergent. Then, using that
$$1=\frac{1}{\Gamma(-s)}\int_0^\infty\big(e^{-t}-1\big)\,\frac{dt}{t^{1+s}},$$
and $(-s)\Gamma(-s)=\Gamma(1-s)$,
we get
\begin{align*}
    \log(-L)u&=\frac{1}{s}\lim_{s\to0}
    \big((-L)^su-u\big) \\
    &=-\lim_{s\to0}\frac{1}{\Gamma(1-s)}
    \bigg[\int_0^\infty\big(e^{tL}u-u\big)\,\frac{dt}{t^{1+s}}
    -\int_0^\infty\big(e^{-t}u-u\big)\,\frac{dt}{t^{1+s}}\bigg]
\end{align*}
from which \eqref{eq:logslimit} follows.
\end{proof}

\begin{proposition}[Formula for $\log(-L)$ for generators of exponentially stable $C_0$-semigroups]\label{prop:log-s}
    Let $L$ be the infinitesimal generator an exponentially stable $C_0$-semigroup $\{e^{tL}\}_{t\geq0}$ on $X$. Then, for any $u\in X$,
    \begin{equation}\label{eq:log-slimit}
    -\log(-L)u=\lim_{s\to0}\frac{1}{\Gamma(1+s)}\int_0^\infty\frac{e^{tL}u-e^{-t}u}{t^{1-s}}\,dt.
    \end{equation}
    In particular, if $u\in D(L)$ then
    $$-\log(-L)u=\int_0^\infty\frac{e^{tL}u-e^{-t}u}{t}\,dt.$$
\end{proposition}

\begin{proof}
Since the $C_0$-semigroup $\{e^{tL}\}_{t\geq0}$ is exponentially stable, it follows that $0\in\rho(-L)$
and, in particular, $(-L)^{-1}$ is well-defined and bounded.
Moreover, for $s>0$, the negative fractional powers of $-L$ can be written as
$$(-L)^{-s}u=\frac{1}{\Gamma(s)}\int_0^\infty e^{tL}u\,\frac{dt}{t^{1-s}}$$
for any $u\in X$, where the integral is absolutely convergent. Then, using that
$$1=\frac{1}{\Gamma(s)}\int_0^\infty e^{-t}\,\frac{dt}{t^{1-s}},$$
and $s\Gamma(s)=\Gamma(1+s)$,
we get
\begin{align*}
    -\log(-L)u&=\frac{1}{s}\lim_{s\to0}
    \big((-L)^{-s}u-u\big) \\
    &=\lim_{s\to0}\frac{1}{\Gamma(1+s)}
    \bigg[\int_0^\infty e^{tL}u\,\frac{dt}{t^{1-s}}
    -\int_0^\infty e^{-t}u\,\frac{dt}{t^{1-s}}\bigg]
\end{align*}
from which \eqref{eq:log-slimit} follows.

Next, we show that if $u\in D(L)$ then we can pass to the limit in \eqref{eq:log-slimit} by applying the dominated convergence theorem. Indeed, for any $0<t<1$, by adding and subtracting $u$,
$$\frac{\|e^{tL}u-e^{-t}u\|_X}{t^{1-s}}
\leq \frac{(M\|Lu\|_X+\|u\|_X)t}{t^{1-s}}
\leq C\in L^1(0,1).$$
On the other hand, by the exponential decay of the semigroup, for any $t\geq1$ and $0<s<1$,
$$\frac{\|e^{tL}u-e^{-t}u\|_X}{t^{1-s}}
\leq \big(Me^{-\varepsilon t}+e^{-t}\big)\|u\|_X\in L^1(1,\infty).$$
The Proposition is proved.
\end{proof}

\section{Examples}\label{section:examples}

In this section we present a series of examples for which our main results hold. All of the operators $L$ involved in the examples will be generators of uniformly bounded or exponentially stable $C_0$-semigroups on appropriate Banach spaces $X$. We will consider examples where $0$ is in the resolvent set of $-L$ (and therefore, more can be said in terms of the equation $\log(-L)u=(\log\lambda)u$) and cases where $0\notin\rho(-L)$. 

Throughout this section, $\Omega\subset\mathbb{R}^n$, $n\geq1$,
is a bounded, Lipschitz domain.

\subsection{Divergence form operators on bounded domains}\label{subsection:divergenceform}

Here we follow \cite{Stinga-Caffa}.
Consider a divergence form linear operator $Lu = \operatorname{div}(a(x)\nabla u)$, where the matrix of coefficients $a(x)=(a^{ij}(x))_{i,j=1}^n$ is symmetric $a^{ij}(x) = a^{ji}(x)$, bounded, measurable and uniformly elliptic, that is, there is $\lambda>0$ such that $a^{ij}(x)\xi_i\xi_j\geq\lambda|\xi|^2$ for all $\xi\in\mathbb{R}^n$ and almost every $x\in\Omega$.

We first consider $L$ subject to homogeneous Dirichlet boundary condition $u\big|_{\partial\Omega} =0$ in the Sobolev space $H^1_0(\Omega)$ of functions $u\in L^2(\Omega)$ such that $\nabla u\in L^2(\Omega)$ and $u=0$ on $\partial\Omega$ in the sense of traces. The family of eigenfunctions $\{\phi_k\}_{k\geq1}\subset H^1_0(\Omega)$ of $-L$ are weak solutions to
$$\begin{cases}
-L\phi_k=\lambda_k\phi_k&\hbox{in}~\Omega\\
\phi_k=0&\hbox{on}~\partial\Omega
\end{cases}$$
and form an orthonormal basis of $L^2(\Omega)$. Here the eigenvalues $\lambda_k$, $k\geq1$, are all positive, with $0<\lambda_1<\lambda_2\leq\cdots\leq\lambda_k\leq\cdots\to\infty$. Then $-L$ is an unbounded operator on $L^2(\Omega)$.
The Lax--Milgram Theorem implies that $0\in\rho(-L)$. This can also be seen as follows: if $f\in L^2(\Omega)$ then $f=\sum_{k=1}^\infty f_k\phi_k$, where the Fourier coefficients of $f$ are $f_k=\langle f,\phi_k\rangle_{L^2(\Omega)}$, and then $u=\sum_{k=1}^\infty\frac{f_k}{\lambda_k}\phi_k$ is the unique weak solution to $-Lu=f$ in $\Omega$ with $u=0$ on $\partial\Omega$. The semigroup generated by $L$ is given by
$$e^{tL}u(x)=\sum_{k=1}^\infty e^{-t\lambda_k}u_k\phi_k(x)$$
where $u_k=\langle u,\phi_k\rangle_{L^2(\Omega)}$. It is clear that the $C_0$-semigroup $\{e^{tL}\}_{t\geq0}$ in $L^2(\Omega)$ is exponentially stable, with $\|e^{tL}u\|_{L^2(\Omega)}\leq e^{-t\lambda_1}\|u\|_{L^2(\Omega)}$.

The fractional power operators $(-L)^{\pm s}$, $s>0$, are defined in the domain
$$D((-L)^{\pm s}) = \Big\{ u\in L^2(\Omega) : \sum_{k=1}^{\infty}|\lambda_k^{\pm s}\langle  u,\phi_k\rangle_{L^2(\Omega)}|^2 < \infty \Big\}$$
as
$$(-L)^{\pm s}u = \sum_{k=1}^\infty\lambda_k^{\pm s}u_k \phi_k\in L^2(\Omega).$$
The heat kernel $W_t^D(x,z)$ for $-L$ with Dirichlet  boundary condition is given by
$$W_t^D(x,z):= \sum_{k=1}^{\infty}e^{-t\lambda_k}\phi_k(x)\phi_k(z)$$
in the sense that
$$\langle e^{tL}u,v\rangle_{L^2(\Omega)}=
\int_\Omega\int_\Omega W_t^D(x,z)u(z)v(x)\,dz\,dx.$$
For $0<s<1$ and $u,v\in D((-L)^s)$, manipulation of the semigroup formula 
$$(-L)^su = \frac{1}{\Gamma(-s)}\int_0^{\infty}(e^{tL}u - u)\frac{dt}{t^{1+s}}$$
yields the following integro-differential pointwise formula
\begin{multline*}
  \langle (-L)^su,v\rangle_{L^2(\Omega)}\\
 =\int_{\Omega}\int_{\Omega}(u(x)-u(z))(v(x)-v(z))K^D_s(x,z)\,dx\,dz + \int_{\Omega}u(x)v(x)B^D_s(x)\,dx 
\end{multline*}
with
$$0\le K^D_s(x,z ):= \frac{1}{2|\Gamma(-s)|}\int_0^{\infty}W^D_t(x,z)\frac{dt}{t^{1+s}}\leq \frac{c_{n,s}}{|x-z|^{n+2s}},\quad x\neq z,$$
and
$$B^D_s(x):= \frac{1}{2|\Gamma(-s)|}\int_0^{\infty}(1-e^{tL}1(x))\frac{dt}{t^{1+s}}\geq 0.$$
The proofs for the pointwise formula and kernel estimates, along with a treatment of the functional analytic setting for the fractional powers can be found in \cite{Stinga-Caffa}.

Next, the logarithm $\log(-L)$ of $-L$ has domain
$$D(\log(-L))=\Big\{ u\in L^2(\Omega)  : \sum_{k=1}^\infty |(\log\lambda_k)\langle u,\phi_k\rangle_{L^2(\Omega)}|^2<\infty\Big\}$$
and is given by
$$\log(-L)u(x)=\sum_{k=1}^\infty(\log\lambda_k)u_k\phi_k(x)\in L^2(\Omega).$$
Note that $D((-L)^s)\subset D(\log(-L))$ since, for $s>0$, $\lim_{r\rightarrow \infty}\frac{\log r}{r^s} =0.$
For $u,v \in D(\log(-L))$, we can write
\begin{align*}
    \langle \log(L)u,v \rangle_{L^2(\Omega)} &= \sum_{k=1}^\infty (\log\lambda_k)u_kv_k \\
    &=\sum_{k=1}^\infty \int_0^{\infty}\big[e^{-t}u_kv_k-e^{-t\lambda_k}u_kv_k\big]\,\frac{dt}{t} \\
    &=\int_0^{\infty}\bigg[e^{-t}\sum_{k=1}^\infty u_kv_k - \sum_{k=1}^\infty e^{-t\lambda_k}u_kv_k\bigg]\, \frac{dt}{t} \\
    &=\int_0^{\infty}\big(e^{-t}\langle u,v\rangle_{L^2(\Omega)} - \langle e^{tL}u,v\rangle_{L^2(\Omega)}\big)\,\frac{dt}{t}
    \\&=\int_0^\infty\langle e^{-t}u-e^{tL}u,v\rangle_{L^2(\Omega)}\,\frac{dt}{t}.
\end{align*}
The interchange of the sum and the integral above is justified because the eigenvalues of $-L$ form a strictly positive, unbounded sequence, so there exists $n\in \mathbb{N}$ such that $\log(\lambda_m) \geq 1$ for all $m\geq n$. Then,
\begin{align*}
    \sum_{k=1}^{\infty}\int_0^{\infty}|e^{-t}-e^{-t\lambda_k}||u_k||v_k|\,\frac{dt}{t} &= C + \sum_{k=n}^{\infty}|u_k||v_k|\int_0^{\infty}(e^{-t}-e^{-t\lambda_k})\,\frac{dt}{t}\\
    &= C + \sum_{k=n}^{\infty}\log(\lambda_k)|u_k||v_k|\\
    &\leq C + \sum_{k=n}^{\infty}\log(\lambda_k)^2(|u_k|^2 + |v_k|^2),
\end{align*}
and the last term is finite if $u,v \in D(\log(-L))$.

Continuing with the Dirichlet boundary condition case, it has been shown (see, for instance, \cite{ArendtTerElst1997}) that the heat kernel $W_t^D(x,z)$ of the semigroup $e^{tL}$satisfies a Gaussian upper-bound
$$0\leq W_t^D(x,z)\le ct^{-n/2}e^{-b\frac{|x-z|^2}{t}}$$
for some $b,c>0$, for all $t>0$ and $x,z\in\Omega$.
Since $-L$ is nonnegative and selfadjoint, the Gaussian bound on its heat kernel guarantees that the semigroup extends to a uniformly bounded $C_0$-semigroup on $L^p(\Omega)$, see \cite[Chapter 7]{Ouhabaz2005}.  In fact, for $1\leq p < \infty$, if we let $-L_p$ be the infinitesimal generator of $\{e^{tL}\}_{t\geq0}$ on $L^p(\Omega)$, then $\sigma(-L) = \sigma(-L_p)$, see \cite[Theorem 7.10]{Ouhabaz2005}.  In particular, $0\in \rho (-L_p)$.

We next consider $L$ subject to homogeneous Neumann boundary condition $a(x)\cdot\nabla u\big|_{\partial\Omega} =0$ in the Sobolev space $H^1_\ast(\Omega)$ of functions $u\in L^2(\Omega)$ such that $\nabla u\in L^2(\Omega)$ and $\int_\Omega u\,dx=0$. The family of eigenfunctions $\{\psi_k\}_{k\geq1}\subset H^1_\ast(\Omega)$ of $-L$ are weak solutions to
$$\begin{cases}
-L\psi_k=\mu_k\psi_k&\hbox{in}~\Omega\\
a(x)\cdot\nabla\psi_k=0&\hbox{on}~\partial\Omega
\end{cases}$$
and form an orthonormal basis of $L^2_\ast(\Omega)=\{u\in L^2(\Omega):\int_\Omega u\,dx=0\}$. In $L^2_\ast(\Omega)$, the eigenvalues
$\{\mu_k\}_{k\geq1}$ are positive, with
$0<\mu_1<\mu_2\leq\cdots\leq\mu_k\leq\cdots\to\infty$. Once again, the Lax--Milgram Theorem (or the spectral representation of $-L$ in terms of eigenfunctions) give that $0\in\rho(-L)$. The semigroup generated by $L$ in $L^2_\ast(\Omega)$ is given by
$$e^{tL}u(x)=\sum_{k=1}^\infty e^{-t\mu_k}u_k\psi_k(x) = \int_\Omega W_t^N(x,z)u(z)\,dz$$
with $u_k=\langle u,\psi_k\rangle_{L^2(\Omega)}$ and
$$W_t^N(x,z)=\sum_{k=1}^\infty e^{-t\mu_k}\psi_k(x)\psi_k(z)$$
being the corresponding heat kernel. Clearly, $\{e^{tL}\}_{t\geq0}$ is an exponentially stable $C_0$-semigroup on $L^2_\ast(\Omega)$.

Similarly to the Dirichlet case, manipulation of the semigroup formula for fractional powers $0<s<1$ yields 
$$\langle (-L)^su,v\rangle_{L^2(\Omega)} = \int_{\Omega}\int_{\Omega}(u(x)-u(z))(v(x)-v(z))K_s^N(x,z)\,dx\,dz$$
with $$0\le K_s^N(x,z ):= \frac{1}{2|\Gamma(-s)|}\int_0^{\infty}W^N_t(x,z)\frac{dt}{t^{1+s}}\leq \frac{c_{\Omega,n,s}}{|x-z|^{n+2s}}, \ \ x,z\in \Omega.$$
Note that in this case, as opposed to the Dirichlet case, $e^{tL}1(x)\equiv 1$
(when considering the semigroup in $L^2(\Omega)$), which is why there is no zero order term as in the Dirichlet case. The fractional powers of $-L$, as well as $\log(-L)$ can be defined similarly to the Dirichlet case on the corresponding Hilbert spaces, always assuming that $\int_\Omega u\,dx=0$.

It is worth noting that if we were to consider $L$ on the domain $H^1(\Omega)$ then $0 \notin \rho(L)$, but we would be able to extend the semigroup, generated by $L$ on $L^2(\Omega)$ to a semigroup of positivity-preserving contractions on $L^p(\Omega)$, $1\leq p < \infty.$  Using the uniform ellipticity assumption, an obvious modification to the proof in \cite[Theorem 13.48]{vanNeerven_2022} gives that the semigroup preserves positivity, that is, $e^{tL}u\geq 0 $ almost everywhere if $u\geq 0$ almost everywhere. This implies that $\{e^{tL}\}_{t\geq0}$ is doubly submarkovian, since $e^{tL}1\equiv 1$.  Then \cite[Theorem 13.50]{vanNeerven_2022} applies to $L$, and we have the result. 

\subsection{Laplace--Beltrami operators on closed Riemannian manifolds}

Let $(M,g)$ be an $n$-dimensional, closed Riemannian manifold, with $n\geq 2$. The usual Laplace-Beltrami operator $\Delta_g$ admits a sequence of eigenfunctions $\{\phi_k\}_{k\geq1}$ which can be taken as an orthonormal basis of $L^2(M)$ along with corresponding eigenvalues $0=\lambda_1 \leq \lambda_2 \leq \cdots\leq\lambda_k\leq\cdots\to \infty$.  The semigroup generated by $\Delta_g$ is defined spectrally in the usual way \cite{Rosenberg1997}.  We restrict to $H^1_*(M)$, the Hilbert-Sobolev space on $M$ of zero-mean functions on $M$.  In this setting, the fractional powers and the logarithm of $-\Delta_g$ can be defined spectrally in a way analogous to the divergence form operator with Neumann boundary condition.

\subsection{Nondivergence form operators in bounded domains} 

 We consider the Banach space of continuous functions in $\overline{\Omega}$ that vanish on the boundary:
 $$C_0(\Omega):=\{u\in C(\overline{\Omega}): \left. u \right|_{\partial\Omega}=0\}.$$
 We define the nondivergence form elliptic operator
 $$Lu=a^{ij}(x)\partial_{ij}u\qquad\hbox{in}~\Omega$$
 with domain
 $D(L)=\{u\in C_0(\Omega)\cap W^{2,n}_{\mathrm{loc}}(\Omega): Lu\in C_0(\Omega)\}$, where $W^{2,n}_{\mathrm{loc}}(\Omega)$ is the set of functions $u\in L^n_{\mathrm{loc}}(\Omega)$ such that their second order derivatives $\partial_{ij}u$ are also in $L^n_{\mathrm{loc}}(\Omega)$. Here the coefficients $a^{ij}(x) \in C(\Omega)\cap L^{\infty}(\Omega)$ are symmetric and uniformly elliptic. Under these conditions, $L$ generates a uniformly bounded $C_0$-semigroup  $\{e^{tL}\}_{t\geq 0}$ in $C_0(\Omega)$, which is also exponentially stable, see \cite[Proposition~4.7]{ArendtSchaetzle2014}. 

Since the semigroup is exponentially stable, it follows that $0\in\rho(L)$. Another way of showing the latter is as follows: if $Lu=0$ in $\Omega$ then, from the Alexandroff--Bakelman--Pucci (ABP) estimate \cite{gilbarg2001elliptic} and the fact that $u\in C_0(\Omega)$, it must be that $u=0$ in $\Omega$.
For $f\in C_0(\Omega)$, there is a solution to $Lu =f$, see \cite{ArendtSchaetzle2014} or \cite{gilbarg2001elliptic}. Then, as a corollary to the ABP estimate, we have 
$$\|u\|_{C_0(\Omega)}=\|u\|_{L^{\infty}(\Omega)}\leq \|u\|_{L^{\infty}(\partial\Omega)} + C\|f\|_{L^{\infty}(\Omega)}=C\|f\|_{L^{\infty}(\Omega)},$$
so the inverse of $L$ is bounded and, thus, $0\in\rho(L)$.

 The fractional powers $(-L)^s$ can be defined using the method of semigroups, see, for instance, \cite{MR220096}. For example, for $0<s<1$, we have (see \cite{STINGA2021245} and \cite{MR4914523})
$$(-L)^su(x) = \lim_{\varepsilon\to 0}\frac{1}{\Gamma(-s)}\int_{\varepsilon}^{\infty}(e^{tL}u(x)-u(x))\frac{dt}{t^{1+s}}\qquad\hbox{in}~C_0(\Omega).$$
In fact, it is known that $u\in D((-L)^s)$ if and only if the above limit exists. This formula can be extended for any $s>0$ accordingly.

The logarithm of $-L$ is defined by the semigroup formula
$$ \log(-L)u(x) = \lim_{s\rightarrow 0}\frac{1}{\Gamma(1+s)}\int_0^{\infty}\frac{e^{-t}u(x) - e^{tL}u(x)}{t}\,dt\qquad\hbox{in}~C_0(\Omega), $$
where we impose that $u\in D(\log(-L))$ if and only if the limit above exists. 

\subsection{Master operators in bounded domains}

This example follows from \cite{B-S-1, B-DLC-S}, wherein more details and regularity theory can be found. We consider the parabolic operator $$Lu= -\partial_tu +\operatorname{div} (a(x)\nabla u)$$ where $a(x)$ is symmetric, bounded, measurable and uniformly elliptic as in Subsection \ref{subsection:divergenceform}, acting on
functions $u(t,x)\in L^2(\mathbb{R} \times \Omega)$ subject to, say, homogeneous Dirichlet boundary conditions $u(t,x)=0$ for $x\in\partial\Omega$ and all $t\in\mathbb{R}$ (the homonogeneous Neumann boundary condition can also be considered in a parallel way to the following analysis by adding the restriction that all functions $u$ have zero average with respect to $x$ for all $t$).  

Using the eigenfunction decomposition of the elliptic operator $\operatorname{div}(a(x)\nabla)$ on $H^1_0(\Omega)$ as described in Subsection \ref{subsection:divergenceform} and the Fourier transform in $\mathbb{R}$, given a function $u\in L^2(\mathbb{R}\times \Omega)$ we can write:
$$u(t,x)= \frac{1}{(2\pi)^{1/2}}\int_{-\infty}^{\infty}\sum_{k=1}^{\infty}\widehat{u_k}(\rho)\phi_k(x)e^{it\rho}\,d\rho,$$
where
$$u_k(t) = \int_{\Omega}u(t,x)\phi_k(x)\,dx\qquad t\in\mathbb{R}$$
and $\widehat{u_k}(\rho)$ is the Fourier transform of $u_k(t)$ with respect to the variable $t\in\mathbb{R}$:
$$\widehat{u_k}(\rho)=\frac{1}{(2\pi)^{1/2}}\int_{-\infty}^\infty u_k(t)e^{-i\rho t}\,dt\qquad\rho\in\mathbb{R}.$$
Then
$$-Lu = \frac{1}{\sqrt{2\pi}}\int_{-\infty}^{\infty}\sum_{k=1}^{\infty}(i\rho + \lambda_k)\widehat{u_k}(\rho)\phi_k(x)e^{it\rho}d\rho.$$
The semigroup generated by $L$ in $L^2(\mathbb{R}\times\Omega)$ is given by
$$e^{\tau L}u(t,x)=
\sum_{k=1}^\infty e^{-t\lambda_k}u_k(t-\tau)\phi_k(x).$$
The semigroup is exponentially stable on $L^2(\mathbb{R}\times\Omega)$. In particular, $0\in\rho(L)$. Another way of proving this is as follows. For $f\in L^2(\mathbb{R}\times \Omega)$, we can choose
$$u(t,x)=\frac{1}{(2\pi)^{1/2}}\int_{-\infty}^{\infty}\sum_{k=1}^{\infty}\frac{1}{i\rho + \lambda_k}\widehat{f_k}(\rho)\phi_k(x)e^{it\rho}\,d\rho$$ 
as a solution to $-Lu=f$, and uniqueness follows from the Dirichlet boundary condition and the parabolic maximum principle.

The fractional power operators $(-L)^{\pm s}$, $s>0$ are defined in the domain
$$D((-L)^{\pm s}):= \Big\{ u\in L^2(\mathbb{R}\times \Omega) : \int_{\mathbb{R}}\sum_{k=1}^\infty|i\rho + \lambda_k|^{\pm 2s}|\widehat{u_k}(\rho)|^2 d\rho < \infty \Big\}$$
by means of
$$(-L)^{\pm s}u(t,x) = \frac{1}{\sqrt{2\pi}}\int_{-\infty}^{\infty}\sum_{k=1}^{\infty}(i\rho + \lambda_k)^{\pm s}\widehat{u_k}(\rho)\phi_k(x)e^{it\rho}d\rho.$$
Semigroup and pointwise formulas can be deduced for these operators, see \cite{B-S-1, B-DLC-S}.

The logarithm of $-L$ is defined using the principal branch of the complex logarithm. Define the domain
$$D(\log(-L))=\Big\{ u\in L^2(\mathbb{R}\times \Omega) : \int_{-\infty}^\infty\sum_{k=1}|\log(i\rho + \lambda_k)|^2|\widehat{u_k}(\rho)|^2 \,d\rho < \infty \Big\},$$
Then, in $L^2(\mathbb{R}\times\Omega)$,
$$\log(-L)u(t,x)= \frac{1}{(2\pi)^{1/2}}\int_{-\infty}^\infty\sum_{k=1}^{\infty}\log(i\rho+\lambda_k)
\widehat{u_k}(\rho)\phi_k(x)e^{it\rho}\,d\rho.$$
Again, semigroup and pointwise formulas for this operator can be found using the results of Section \ref{section:log} and the heat kernel associated to $\operatorname{div}(a(x)\nabla)$.
    
\subsection{The discrete Laplacian}

This example follows the work in \cite{MR3787555} in dimension $n=1$. The generalization to dimensions $n\geq2$ follows the same pattern, see \cite{MR3666807}.

We consider a mesh of size $h>0$ on $\mathbb{Z}$ given by $\mathbb{Z}_h = \{hn : n\in \mathbb{Z}\} $. The action of the discrete Laplacian $\Delta_h$ on this mesh is given, for a sequence $u:\mathbb{Z}_h \rightarrow \mathbb{R}$, by
$$\Delta_h u_j = \frac{1}{h^2}(u_{j+1}-2u_j + u_{j-1}),$$
where $u_j=u(hj)$ is the value of $u$ at the mesh point $hj$.
The solution to the semidiscrete heat equation with initial condition $u$,
\[ \begin{cases} 
      \partial_tw_j(t) = \Delta_hw_j(t) & \hbox{in}~\mathbb{Z}_h \times (0,\infty) \\
      w_j(0) =u_j & \hbox{on}~\mathbb{Z}_h \\
      
   \end{cases}
\]
is given by convolution with the semidiscrete heat kernel as
$$w_j(t)=e^{t\Delta_h}u_j=\sum_{m\in\mathbb{Z}}G(j-m,\tfrac{t}{h^2})u_m,$$
where $$G(m,t)=e^{-2t}I_m(2t) \qquad \hbox{for}~ m\in \mathbb{Z}$$
and
$$I_j(t) = \sum_{m=0}^{\infty}\frac{1}{m!\Gamma(m+j+1)}\Big(\frac{t}{2}\Big)^{2m+j}$$
is the modified Bessel function of order $j$. It is clear that $I_0(0)=1$ and $I_k(0)=0$ for $k\not= 0$, and $$\sum_{m\in\mathbb{Z}}e^{-2m}I_m(2t) =1.$$ for $t>0$.  These facts, along with the discrete Young's convolution inequality, imply that $\{e^{t\Delta_h}\}_{t\geq0}$ defines a uniformly bounded $C_0$-semigroup on $\ell^p(\mathbb{Z}_h)$ for any $1\leq p\leq\infty$, see \cite{MR3666807}.
Considered as an operator on $\ell^2(\mathbb{Z}_h)$, $0\not \in \rho(\Delta_h),$ since $$\widehat{\Delta_hu}(\xi)= -\frac{4}{h^2}\sin^2\Big(\frac{h\xi}{2}\Big)\widehat{u}(\xi)\qquad \hbox{for}~\xi\in[-\pi,\pi].$$

For $0<s<1$ the fractional discrete Laplacian $(-\Delta_h)^s$ can be defined via the semigroup formula on the domain $$\ell_s:=\Bigg\{u: \mathbb{Z}_h\rightarrow \mathbb{R}: \|u\|_{\ell_s}:= \sum_{m\in\mathbb{Z}} \frac{|u_m|}{(1+|m|)^{1+2s}}<\infty\Bigg\}$$
as
\begin{align*}
(-\Delta_h)^su_j&= \frac{1}{\Gamma(-s)}\int_0^{\infty}(e^{t\Delta_h}u_j - u_j)\frac{dt}{t^{1+s}}\\
&=\frac{1}{h^{2s}}\sum_{m\not=j}(u_j-u_m)K_s(j-m),
\end{align*}
where $$K_s(m)=\frac{1}{|\Gamma(-s)|}\int_0^{\infty}G(m,t)\frac{dt}{t^{1+s}}.$$
We can also define $\log(-\Delta_h)$ via the semigroup formula:
$$\log(-\Delta_h)u_j=\int_0^{\infty}\frac{e^{-t}u_j-e^{t\Delta_h}u_j}{t}dt.$$

\subsection{The Laplacian on $L^{p}(\mathbb{R}^n)$}

Fix $1<p<\infty$.  We consider the Laplacian $\Delta$ on $L^p(\mathbb{R}^n)$ with domain $D(\Delta)$ the Sobolev space $W^{2,p}(\mathbb{R}^n)$. Then $\Delta$ is the infinitesimal generator of the classical heat semigroup
$$e^{t\Delta}u(x)=\frac{1}{(4\pi t)^{n/2}}\int_{\mathbb{R}^n}e^{-|x-y|^2/4t}u(y)\,dy\qquad\hbox{in}~L^p(\mathbb{R}^n)$$
which is a uniformly bounded $C_0$-semigroup. Furthermore, the spectrum of the Laplacian in $L^p(\mathbb{R}^n)$ is $\sigma(-\Delta) = [0,\infty)$. See, for example, \cite[Theorem 2.3.3]{martinez2001fractional}.

The fractional powers of the Laplacian on $L^{p}(\mathbb{R}^n)$ are defined using the semigroup formula as in \cite{MR220096}. For example, for $0<s<1$,
$$(-\Delta)^su(x) = \lim_{\varepsilon\rightarrow 0}\frac{1}{\Gamma(-s)}\int_{\varepsilon}^{\infty}(e^{t\Delta}u(x)-u(x))\,\frac{dt}{t^{1+s}}\qquad\hbox{in}~L^p(\mathbb{R}^n).$$
It turns out that the domain of $(-\Delta)^s$ in $L^p(\mathbb{R}^n)$ is the Bessel potential space $(I-\Delta)^{-s}(L^p(\mathbb{R}^n))$, see \cite[Theorem 7.16]{samko2001hypersingular}. The logarithm of the Laplacian is also defined by the semigroup formula
$$\log(-\Delta)u = \lim_{s\rightarrow 0}\frac{1}{\Gamma(1+s)}\int_0^{\infty}\frac{e^{-t}u - e^{t\Delta}u}{t}\,dt$$
with domain the set of functions $u$ for which the limit exists in $L^p(\mathbb{R}^n)$.

\subsection{The left derivative on $L^p(\mathbb{R})$}

This example follows the works \cite{MR3456835,STINGA2020111505}. We consider the left derivative operator
$$D_{\mathrm{left}}u(t)=\lim_{\tau\to0^+}
\frac{u(t-\tau)-u(t)}{\tau}$$
on $L^p(\mathbb{R})$, $1<p<\infty$, with domain the Sobolev space $W^{1,p}(\mathbb{R})$. In this case $0\not \in \rho(D_{\text{left}})$. To see this, let 
$$u_n(x) = \frac{1}{(2n)^{1/p}}e^{-|x|/np}.$$ Then, $\|u_n\|_{L^p(\mathbb{R})}=1,$ but
$$\lim_{n\to \infty}\|D_{\mathrm{left}}u_n\|_{L^p(\mathbb{R})} =0,$$ so the inverse of $D_{\mathrm{left}}$ cannot be bounded (see \cite{martinez2001fractional}). The left derivative is the infinitesimal generator of the semigroup of left translations
$$e^{\tau D_\mathrm{left}}u(t)=u(t-\tau).$$
It is clear that $\{e^{\tau D_{\mathrm{left}}}\}_{\tau\geq0}$ defines a uniformly bounded $C_0$-semigroup on $L^p(\mathbb{R})$.

The pointwise formula for the fractional power of the left derivative, or the Fourier--Marchaud left fractional derivative, for $0<s<1$, follows easily from the semigroup formula:
$$(-D_{\mathrm{left}})^su(x) =
\frac{1}{\Gamma(-s)}\int_{-\infty}^x\frac{u(\tau)-u(t)}{(t-\tau)^{1+s}}\,dt.$$
Similarly, one can define the Marchaud right fractional derivative by considering the semigroup of right translations.

We define $\log(-D_{\mathrm{left}})$ as 
$$\log(-D_{\mathrm{left}})u(t) = \lim_{s\rightarrow 0}\frac{1}{\Gamma(1-s)}\int_0^{\infty}\frac{e^{-\tau}u(t)-e^{-\tau D_{\mathrm{left}}}u(t)}{\tau^{1+s}}\,d\tau$$
where $u\in D(-\log(D_{\mathrm{left}}))$ if and only if the limit above exists.

\subsection{The heat operator}

We consider the heat operator
$$Lu =-\partial_tu+\Delta u$$
for functions $u(t,x)$ in $L^2(\mathbb{R}^{n+1})$, $t\in\mathbb{R}$, $x\in\mathbb{R}^n$, as in \cite{StingaTorrea2017}.
We have the Fourier transform identity
$$\widehat{-Lu}(\rho,\xi) =  (i\rho + |\xi|^2)\widehat{u}(\rho, \xi)$$ 
for $\rho \in \mathbb{R}$ and $\xi \in \mathbb{R}^n$. In particular, $0 \not \in \rho(L)$. The semigroup generated by $L$ is given by
$$e^{\tau L}u(t,x) =\frac{1}{(4\pi \tau)^{n/2}}\int_{\mathbb{R}^n}e^{-|y|^2/(4\tau)}u(t-\tau,x-y)\,dy.$$
The Fourier transform shows that $\{e^{\tau L}\}_{\tau\geq0}$ defines a uniformly bounded $C_0$-semigroup on $L^2(\mathbb{R}^{n+1})$.

For $0<s<1$, the fractional powers $(-L)^s$ are defined on the domain
$$D((-L)^s)=\big\{ u(t,x)\in L^2(\mathbb{R}^{n+1}) : (i\rho + |\xi|^2)^s\widehat{u}(\rho,\xi)\in L^2(\mathbb{R}^{n+1})\big\}$$ with the semigroup formula,
$$(-L)^su(t,x) = \frac{1}{\Gamma(-s)}
\int_{0}^{\infty} \big(e^{\tau L}u(x,t) - u(x,t)\big)\,\frac{d\tau}{\tau^{1+s}}.$$
The logarithm of $-L$ is also defined via the semigroup formula
$$\log(-L)u(t,x) = \lim_{s\rightarrow 0}\frac{1}{\Gamma(1-s)}\int_0^{\infty}\frac{e^{-\tau}u(t,x)-e^{\tau L}u(t,x)}{\tau^{1+s}}\,d\tau$$
whenever the limit exists.


\bibliographystyle{amsplain}
\bibliography{reference.bib}

\def\ocirc#1{\ifmmode\setbox0=\hbox{$#1$}\dimen0=\ht0 \advance\dimen0
  by1pt\rlap{\hbox to\wd0{\hss\raise\dimen0
  \hbox{\hskip.2em$\scriptscriptstyle\circ$}\hss}}#1\else {\accent"17 #1}\fi}
  \def\cprime{$'$} \def\cprime{$'$}
  \def\ocirc#1{\ifmmode\setbox0=\hbox{$#1$}\dimen0=\ht0 \advance\dimen0
  by1pt\rlap{\hbox to\wd0{\hss\raise\dimen0
  \hbox{\hskip.2em$\scriptscriptstyle\circ$}\hss}}#1\else {\accent"17 #1}\fi}
\providecommand{\bysame}{\leavevmode\hbox to3em{\hrulefill}\thinspace}
\providecommand{\MR}{\relax\ifhmode\unskip\space\fi MR }
\providecommand{\MRhref}[2]{%
  \href{http://www.ams.org/mathscinet-getitem?mr=#1}{#2}
}
\providecommand{\href}[2]{#2}
\begin{thebibliography}{10}

\bibitem{MR1039339}
Shmuel Agmon, \emph{A representation theorem for solutions of the {H}elmholtz
  equation and resolvent estimates for the {L}aplacian}, Analysis, et cetera,
  Academic Press, Boston, MA, 1990, pp.~39--76. \MR{1039339}

\bibitem{MR466902}
Shmuel Agmon and Lars H\"{o}rmander, \emph{Asymptotic properties of solutions
  of differential equations with simple characteristics}, J. Analyse Math.
  \textbf{30} (1976), 1--38. \MR{466902}

\bibitem{Allen-Caffarelli-Vasseur}
Mark Allen, Luis Caffarelli, and Alexis Vasseur, \emph{A parabolic problem with
  a fractional time derivative}, Arch. Ration. Mech. Anal. \textbf{221} (2017),
  no.~1, 603--630.

\bibitem{Antil_Control}
Harbir Antil and Enrique Ot\'{a}rola, \emph{A fem for an optimal control
  problem of fractional powers of elliptic operators}, SIAM J. Control Optim.
  \textbf{53} (2015), no.~6, 3432--3456.

\bibitem{ArendtSchaetzle2014}
Wolfgang Arendt and Reiner Sch{\"a}tzle, \emph{Semigroups generated by elliptic
  operators in non-divergence form on {$C_0(\Omega)$}}, Ann. Sc. Norm. Super.
  Pisa - Cl. Sci. \textbf{13} (2014), no.~2, 417--434.

\bibitem{ArendtTerElst1997}
Wolfgang Arendt and A.~F.~M. ter Elst, \emph{Gaussian estimates for second
  order elliptic operators with boundary conditions}, J. Oper. Theory.
  \textbf{38} (1997), no.~1, 87--130.

\bibitem{ACM}
Ioannis Athanasopoulos, Luis Caffarelli, and Emmanouil Milakis, \emph{On the
  regularity of the non-dynamic parabolic fractional obstacle problem}, J.
  Differ. Equ. \textbf{265} (2018), no.~6, 2614--2647.

\bibitem{MR220096}
H.~Berens, P.~L. Butzer, and U.~Westphal, \emph{Representations of fractional
  powers of infinitesimal generators of semigroups}, Bull. Amer. Math. Soc.
  \textbf{74} (1968), 191--196. \MR{220096}

\bibitem{BRR}
H~Berestycki, J~M Roquejoffre, and L~Rossi, \emph{The influence of a line with
  fast diffusion on fisher-kpp propagation}, J. Math. Biol. \textbf{66} (2013),
  no.~1, 743--766.

\bibitem{MR3456835}
Ana Bernardis, Francisco~J. Mart\'{\i}n-Reyes, Pablo~Ra\'{u}l Stinga, and
  Jos\'{e}~L. Torrea, \emph{Maximum principles, extension problem and inversion
  for nonlocal one-sided equations}, J. Differential Equations \textbf{260}
  (2016), no.~7, 6333--6362. \MR{3456835}

\bibitem{B-S-1}
A~Biswas and P~R Stinga, \emph{Regularity estimates for nonlocal space-time
  master equations in bounded domains}, J. Evol. Equ. (2020), no.~1.

\bibitem{B-DLC-S}
Animesh Biswas, Marta De~Le\'{o}n-Contreras, and Pablo~Ra\'{u}l Stinga,
  \emph{Harnack inequalities and hölder estimates for master equations}, SIAM
  J. Math. Anal. \textbf{53} (2021), no.~2, 2319--2348.

\bibitem{B-S}
Animesh Biswas and Pablo~Raúl Stinga, \emph{Sharp extension problem
  characterizations for higher fractional power operators in banach spaces}, J.
  Funct. Anal. \textbf{287} (2024), no.~3, 110474.

\bibitem{bonito2018numerical}
Andrea Bonito, Juan~Pablo Borthagaray, Ricardo~H Nochetto, Enrique Ot{\'a}rola,
  and Abner~J Salgado, \emph{Numerical methods for fractional diffusion},
  Comput. Vis. Sci. \textbf{19} (2018), no.~5, 19--46.

\bibitem{boyadzhiev1994logarithms}
Khristo~N. Boyadzhiev, \emph{Logarithms and imaginary powers of operators on
  hilbert spaces}, Collect. Math. \textbf{45} (1994), no.~3, 287--300.

\bibitem{Caffa-Silv}
L~A Caffarelli and L~Silvestre, \emph{An extension problem related to the
  fractional laplacian}, Comm. Partial Diff. Eq. \textbf{32} (2007), no.~1,
  1245--1260.

\bibitem{Caffarelli-Silvestre-Master}
\bysame, \emph{H\"older regularity for generalized master equations with rough
  kernels}, Advances in analysis: the legacy of Elias M. Stein, Princeton Math.
  Ser. \textbf{50} (2014), no.~1, 63--83.

\bibitem{Stinga-Caffa}
L~A Caffarelli and P~R Stinga, \emph{Fractional elliptic equations, caccioppoli
  estimates and regularity}, Ann. Inst. H. Poincar\'e Anal. Non Lin\'eaire
  \textbf{33} (2016), no.~1, 767--807.

\bibitem{CaffarelliSilvestre2009Regularity}
Luis Caffarelli and Luis Silvestre, \emph{Regularity theory for fully nonlinear
  integro-differential equations}, Comm. Pure Appl. Math. \textbf{62} (2009),
  no.~5, 597--638.

\bibitem{CaffarelliSilvestre2011EvansKrylov}
\bysame, \emph{The evans--krylov theorem for nonlocal fully nonlinear
  equations}, Ann. of Math. \textbf{174} (2011), no.~2, 1163--1187.

\bibitem{caffarelli2011regularity}
\bysame, \emph{Regularity results for nonlocal equations by approximation},
  Arch. Ration. Mech. Anal. \textbf{200} (2011), no.~1, 59--88.

\bibitem{CHEN2024110470}
Huyuan Chen and Laurent Véron, \emph{The cauchy problem associated to the
  logarithmic laplacian with an application to the fundamental solution}, J.
  Funct. Anal. \textbf{287} (2024), no.~3, 110470.

\bibitem{Chen-Weth}
Huyuan Chen and Tobias Weth, \emph{The dirichlet problem for the logarithmic
  laplacian}, Comm. Partial Diff. Eq. \textbf{44} (2019), no.~11, 1100--1139.

\bibitem{cheng2022equivalence}
Xinyu Cheng, Dong Li, and Wen Yang, \emph{On the equivalence of classical
  helmholtz equation and fractional helmholtz equation with arbitrary order},
  Commun. Contemp. Math. \textbf{24} (2022), no.~10, 2250036.

\bibitem{CHM}
Laurence Cherfils, Hussein Fakih, and Alain Miranville, \emph{A
  {C}ahn-{H}illiard system with a fidelity term for color image inpainting}, J.
  Math. Imaging Vision \textbf{54} (2016), no.~1, 117--131. \MR{3440241}

\bibitem{MR3666807}
\'Oscar Ciaurri, T.~Alastair Gillespie, Luz Roncal, Jos\'e{}~L. Torrea, and
  Juan~Luis Varona, \emph{Harmonic analysis associated with a discrete
  {L}aplacian}, J. Anal. Math. \textbf{132} (2017), 109--131. \MR{3666807}

\bibitem{MR3787555}
\'{O}scar Ciaurri, Luz Roncal, Pablo~Ra\'{u}l Stinga, Jos\'{e}~L. Torrea, and
  Juan~Luis Varona, \emph{Nonlocal discrete diffusion equations and the
  fractional discrete {L}aplacian, regularity and applications}, Adv. Math.
  \textbf{330} (2018), 688--738. \MR{3787555}

\bibitem{DL}
G~Duvaut and J~L Lions, \emph{Inequalities in mechanics and physics},
  Springer-Verlag, 1976.

\bibitem{dyda2026dirichlet}
Bart{\l}omiej Dyda, Sven Jarohs, and Firoj Sk, \emph{The dirichlet problem for
  the logarithmic {$p$}-laplacian}, Trans. Amer. Math. Soc. \textbf{379}
  (2026), no.~4, 2717--2779.

\bibitem{fall2016liouville}
Mouhamed~Moustapha Fall and Tobias Weth, \emph{Liouville theorems for a general
  class of nonlocal operators}, Potential Anal. \textbf{45} (2016), no.~1,
  187--200.

\bibitem{Gale}
J~E Gale, P~J Miana, and P~R Stinga, \emph{Extension problem and fractional
  operators: semigroups and wave equations}, J.~Evol.~Equ. \textbf{13} (2013),
  343--386.

\bibitem{Ghosh02122017}
Tuhin Ghosh, Yi-Hsuan Lin, and Jingni Xiao, \emph{The calderón problem for
  variable coefficients nonlocal elliptic operators}, Comm. Partial Diff. Eq.
  \textbf{42} (2017), no.~12, 1923--1961.

\bibitem{gilbarg2001elliptic}
David Gilbarg and Neil~S. Trudinger, \emph{Elliptic partial differential
  equations of second order}, 2 ed., Classics in Mathematics, Springer, Berlin,
  Heidelberg, 2001.

\bibitem{GO2}
Guy Gilboa and Stanley Osher, \emph{Nonlocal linear image regularization and
  supervised segmentation}, Multiscale Model. Simul. \textbf{6} (2007), no.~2,
  595--630.

\bibitem{GO}
Guy Gilboa and Stanley Osher, \emph{{Nonlocal Operators with Applications to
  Image Processing }}, Multiscale Model. Simul. \textbf{7} (2008), no.~3,
  1005--1028.

\bibitem{Guan-Murugan}
Vincent Guan, Mathav Murugan, and Juncheng Wei, \emph{Helmholtz solutions for
  the fractional laplacian and other related operators}, Commun. Contemp. Math.
  \textbf{25} (2023), no.~02, 2250016.

\bibitem{hughes1980logarithm}
Rhonda~J. Hughes, \emph{On the logarithm of an abstract potential operator},
  Indiana Univ. Math. J. \textbf{29} (1980), no.~3, 447--454.

\bibitem{JAROHS2020108732}
Sven Jarohs, Alberto Saldaña, and Tobias Weth, \emph{A new look at the
  fractional poisson problem via the logarithmic laplacian}, J. Funct. Anal.
  \textbf{279} (2020), no.~11, 108732.

\bibitem{lee2026fundamental}
David Lee, \emph{Fundamental solutions of the logarithmic laplacian: An
  approach via the division problem}, Proc. Amer. Math. Soc. \textbf{154}
  (2026), 3479--3494.

\bibitem{MR930604}
Celso Mart\'inez, Miguel Sanz, and Luis Marco, \emph{Fractional powers of
  operators}, J. Math. Soc. Japan \textbf{40} (1988), no.~2, 331--347.
  \MR{930604}

\bibitem{martinez2001fractional}
Celso Mart{\'\i}nez~Carracedo and Miguel Sanz~Alix, \emph{The theory of
  fractional powers of operators}, North-Holland Mathematics Studies, vol. 187,
  Elsevier Science, 2001.

\bibitem{Nocheto}
R~H Nochetto, E~Ot\'arola, and A~B Salgado, \emph{A pde approach to fractional
  diffusion in general domains: A priori error analysis}, Found. Comput. Math.
  \textbf{15} (2015).

\bibitem{nollau1969logarithmus}
Volker Nollau, \emph{{\"U}ber den logarithmus abgeschlossener operatoren in
  banachschen r{\"a}umen}, Acta Sci. Math. \textbf{30} (1969), no.~3-4,
  161--174.

\bibitem{MR261381}
\bysame, \emph{\"uber den {L}ogarithmus abgeschlossener {O}peratoren in
  {B}anachschen {R}\"aumen}, Acta Sci. Math. (Szeged) \textbf{30} (1969),
  161--174. \MR{261381}

\bibitem{Ouhabaz2005}
El-Maati Ouhabaz, \emph{Analysis of heat equations on domains}, London
  Mathematical Society Monographs, no.~31, Princeton University Press,
  Princeton, NJ, 2005.

\bibitem{pazy1983semigroups}
Amnon Pazy, \emph{Semigroups of linear operators and applications to partial
  differential equations}, Applied Mathematical Sciences, vol.~44,
  Springer-Verlag, New York-Berlin-Heidelberg-Tokyo, 1983.

\bibitem{Rosenberg1997}
Steven Rosenberg, \emph{The laplacian on a riemannian manifold: An introduction
  to analysis on manifolds}, London Mathematical Society Student Texts,
  vol.~31, Cambridge University Press, Cambridge, 1997.

\bibitem{samko2001hypersingular}
Stefan~G. Samko, \emph{Hypersingular integrals and their applications}, CRC
  Press, 2001.

\bibitem{Stinga-Torrea-SIAM}
P~R Stinga and J~L Torrea, \emph{Regularity theory and extension problem for
  fractional nonlocal parabolic equations and the master equation}, SIAM J.
  Math. Anal. \textbf{49} (2017), no.~1, 3893--3924.

\bibitem{StingaTorrea2017}
Pablo~Ra{\'u}l Stinga and Jos{\'e}~L. Torrea, \emph{Regularity theory and
  extension problem for fractional nonlocal parabolic equations and the master
  equation}, SIAM J. Math. Anal. \textbf{49} (2017), no.~5, 3893--3924.

\bibitem{MR4914523}
Pablo~Ra\'{u}l Stinga and Mary Vaughan, \emph{Interior {S}chauder estimates for
  fractional elliptic equations in nondivergence form}, SIAM J. Math. Anal.
  \textbf{57} (2025), no.~3, 2833--2879. \MR{4914523}

\bibitem{STINGA2020111505}
Pablo~Raúl Stinga and Mary Vaughan, \emph{One-sided fractional derivatives,
  fractional laplacians, and weighted sobolev spaces}, Nonlinear Anal.
  \textbf{193} (2020), 111505.

\bibitem{STINGA2021245}
Pablo~Raúl Stinga and Mary Vaughan, \emph{Fractional elliptic equations in
  nondivergence form: definition, applications and harnack inequality}, J.
  Math. Pures Appl. \textbf{156} (2021), 245--306.

\bibitem{vanNeerven_2022}
Jan van Neerven, \emph{Functional analysis}, Cambridge Studies in Advanced
  Mathematics, Cambridge University Press, 2022.

\bibitem{weilenmann1978continuity}
J{\"u}rg Weilenmann, \emph{Continuity properties of fractional powers, of the
  logarithm, and of holomorphic semigroups}, J. Funct. Anal. \textbf{27}
  (1978), no.~1, 1--20.

\bibitem{yoshikawa1973logarithm}
Atsushi Yoshikawa, \emph{On the logarithm of closed linear operators}, Proc.
  Jpn. Acad. Ser. A Math. Sci. \textbf{49} (1973), no.~3, 169--173.

\end{thebibliography}

\end{document}